\documentclass[12pt]{amsart}%
\usepackage{amsmath}
\usepackage{amsfonts}
\usepackage{amssymb}
\usepackage{graphicx}%
\providecommand{\U}[1]{\protect \rule{.1in}{.1in}}
\newtheorem{theorem}{Theorem}

\newtheorem{corollary}[theorem]{Corollary}

\newtheorem{remark}[theorem]{Remark}

\numberwithin{equation}{section}
\begin{document}
\title[Inequalities for indices of coincidence and entropies]{Inequalities for indices of coincidence and entropies}
\author[Alexandra M\u adu\c ta]{Alexandra M\u adu\c ta}
\address[Alexandra M\u adu\c ta]{Technical University of Cluj-Napoca, Department of
Mathematics, 28 Memorandumului Street, 400114 Cluj-Napoca, Romania}
\email{boloca.alexandra91@yahoo.com}
\author[Diana Otrocol]{Diana Otrocol}
\address[Diana Otrocol]{Technical University of Cluj-Napoca, Department of Mathematics,
28 Memorandumului Street, 400114 Cluj-Napoca, Tiberiu Popoviciu Institute of
Numerical Analysis, Romanian Academy, P.O.Box. 68-1, 400110 Cluj-Napoca, Romania}
\email{Diana.Otrocol@math.utcluj.ro}
\author[Ioan Ra\c sa]{Ioan Ra\c sa}
\address[Ioan Ra\c sa]{Technical University of Cluj-Napoca, Department of Mathematics,
28 Memorandumului Street, 400114 Cluj-Napoca, Romania}
\email{Ioan.Rasa@math.utcluj.ro}

\begin{abstract}
We consider a probability distribution depending on a real parameter $x$. As
functions of $x$, the R\'enyi entropy and the Tsallis entropy can be expressed
in terms of the associated index of coincidence $S(x)$. We establish
recurrence relations and inequalities for $S(x),$ which can be used in order
to get information concerning the two entropies.

\end{abstract}
\maketitle


\section{Introduction}

Let $c\in \mathbb{R}$. Set $I_{c}=\left[  0,-\frac{1}{c}\right]  $ if $c<0$,
and $I_{c}=[0,+\infty)$ if $c\geq0$. For $\alpha \in \mathbb{R}$ and
$k\in \mathbb{N}_{0}$ the binomial coefficients are defined as usual by
\[
{\binom{\alpha}{k}}:=\frac{\alpha(\alpha-1)\dots(\alpha-k+1)}{k!}%
\quad \mbox{if }k\in \mathbb{N},\mbox{ and }{\binom{\alpha}{0}}:=1.
\]

Let $n>0$ be a real number, $k\in \mathbb{N}_{0}$ and $x\in I_{c}$. Define
\[
p_{n,k}^{[c]}(x):=(-1)^{k}{\tbinom{-\frac{n}{c}}{k}}(cx)^{k}(1+cx)^{-\frac
{n}{c}-k},\quad \mbox{ if }c\neq0,
\]%
\[
p_{n,k}^{[0]}(x):=\lim_{c\rightarrow0}p_{n,k}^{[c]}(x)=\frac{(nx)^{k}}%
{k!}e^{-nx},\quad \mbox{ if }c=0.
\]

Then $\sum_{k=0}^{\infty}p_{n,k}^{[c]}(x)=1$. Suppose that $n>c$ if $c\geq0$,
or $n=-cl$ with some $l\in \mathbb{N}$ if $c<0$.

With this notation we consider the discrete distribution of probability
$\left(  p_{n,k}^{[c]}(x)\right)  _{k=0,1,...}$ depending on the parameter
$x\in I_{c}.$

The associated index of coincidence is%
\begin{equation}
S_{n,c}(x):=\sum \limits_{k=0}^{\infty}\left(  p_{n,k}^{[c]}(x)\right)
^{2},\ x\in I_{c}.\label{1.1}%
\end{equation}

The R\'{e}nyi entropy and the Tsallis entropy corresponding to the same
distribution of probability are defined, respectively, by%
\begin{equation}
R_{n,c}(x)=-\log S_{n,c}(x) \label{1.2}%
\end{equation}
and%
\begin{equation}
T_{n,c}(x)=1-S_{n,c}(x). \label{1.3}%
\end{equation}
For $c=-1$ we are dealing with the binomial distribution and%
\begin{equation}
S_{n,-1}(x):=\sum \limits_{k=0}^{n}\left(  {\dbinom{n}{k}}x^{k}(1-x)^{n-k}%
\right)  ^{2},\ x\in \lbrack0,1]. \label{1.4}%
\end{equation}
The case $c=0$ corresponds to the Poisson distribution, for which%
\begin{equation}
S_{n,0}(x):=e^{-2nx}\sum \limits_{k=0}^{\infty}\dfrac{(nx)^{2k}}{\left(
k!\right)  ^{2}},\ x\geq0. \label{1.5}%
\end{equation}
For $c=1$ we have the negative binomial distribution, with%
\begin{equation}
S_{n,1}(x):=\sum \limits_{k=0}^{\infty}\left(  {\dbinom{n+k-1}{k}}%
x^{k}(1+x)^{-n-k}\right)  ^{2},\ x\geq0. \label{1.6}%
\end{equation}

The binomial, Poisson, respectively negative binomial distributions correspond
to the classical Bernstein, Sz\'asz-Mirakyan, respectively Baskakov operators
from Approximation Theory.

The distribution%
\[
\left(  \dbinom{n}{k}x^{k}(1+x)^{-n}\right)  _{k=0,1,\ldots,n},\ x\in
\lbrack0,+\infty),
\]
corresponds to the Bleimann-Butzer-Hahn operators, while%
\[
\left(  \dbinom{n+k}{k}x^{k}(1-x)^{n+1}\right)  _{k=0,1,\ldots},\ x\in
\lbrack0,1),
\]
is connected with the Meyer-K\"{o}nig and Zeller operators.

The indices of coincidence and the entropies associated with all these
distributions were studied in \cite{R1}. We continue this study. To keep the
same notation as in \cite{R1}, let%
\begin{align}
F_{n}(x) &  :=S_{n,-1}(x),\ G_{n}(x):=S_{n,1}(x),\ K_{n}(x):=S_{n,0}%
(x),\nonumber \\
U_{n}(x) &  :=%
{\textstyle \sum \limits_{k=0}^{n}}
\left(  \dbinom{n}{k}x^{k}(1+x)^{-n}\right)  ^{2},\ x\in \lbrack0,+\infty
),\label{1.7}\\
J_{n}(x) &  :=%
{\textstyle \sum \limits_{k=0}^{\infty}}
\left(  \dbinom{n+k}{k}x^{k}(1-x)^{n+1}\right)  ^{2},\ x\in \lbrack
0,1).\label{1.8}%
\end{align}

In Section 2 we present several relations between the functions $F_{n}(x)$,
$G_{n}(x),\ U_{n}(x),\ J_{n}(x)$,  as well as between these functions and the
Legendre polynomials. By using the three-terms recurrence relations involving
the Legendre polynomials we establish recurrence relations involving three
consecutive terms from the sequences $\left(  F_{n}(x)\right)  $, $\left(
G_{n}(x)\right)  ,\newline  \left(  U_{n}(x)\right)  ,\ $respectively $\left(
J_{n}(x)\right)  $. We recall also some explicit expressions of these functions.

Section 3 is devoted to inequalities between consecutive terms of the above
sequences; in particular, we emphasize that for fixed $x$ the four
sequences are convex.

Other inequalities are presented in Section 4. All the inequalities can be
used to get information about the R\'enyi entropies and Tsallis entropies
connected with the corresponding probability distributions.

\section{Recurrence relations}

$F_{n}(x)$ is a polynomial, $G_{n}(x),\ U_{n}(x),\ J_{n}(x)$ are rational
functions. On their maximal domains, these functions are connected by several
relations (see \cite{R1}, Cor. 13, (46), (53), (54)):%
\begin{align}
F_{n}(x)  &  =(1-2x)^{2n+1}G_{n+1}(-x),\label{2.1}\\
F_{n}(x)  &  =U_{n}\left(  \frac{x}{1-x}\right)  ,\label{2.2}\\
F_{n}(x)  &  =-(1-2x)^{2n+1}J_{n}\left(  \frac{x-1}{x}\right)  . \label{2.3}%
\end{align}
Consider the Legendre polynomial (see \cite[22.3.1]{A})%
\begin{equation}
P_{n}(t)=2^{-n}%
{\textstyle \sum \limits_{k=0}^{n}}
\dbinom{n}{k}^{2}(x+1)^{k}(x-1)^{n-k}. \label{2.4}%
\end{equation}
Then (see \cite[(39)]{R1})%
\begin{equation}
P_{n}(t)=(1-2x)^{-n}F_{n}(x), \label{2.5}%
\end{equation}
where%
\[
t=\frac{1-2x+2x^{2}}{1-2x},\ x\in \lbrack0,\frac{1}{2}).
\]

Combining (\ref{2.5}) with (\ref{2.1}), (\ref{2.2}) and (\ref{2.3}), we get%
\begin{align}
P_{n}(t)  &  =(1-2x)^{n+1}G_{n+1}(-x),\label{2.6}\\
P_{n}(t)  &  =(1-2x)^{-n}U_{n}\left(  \frac{x}{1-x}\right), \label{2.7}\\
P_{n}(t)  &  =-(1-2x)^{n+1}J_{n}\left(  \frac{x-1}{x}\right)  . \label{2.8}%
\end{align}
On the other hand, the Legendre polynomials satisfy the recurrence relation
(\cite[22.7.1]{A})%
\begin{equation}
(n+1)P_{n+1}(t)-(2n+1)tP_{n}(t)+nP_{n-1}(t)=0. \label{2.9}%
\end{equation}
Now (\ref{2.9}) together with (\ref{2.5}), (\ref{2.6}), (\ref{2.7}),
(\ref{2.8}) lead to

\begin{theorem}
\label{th2.1} The functions $F_{n}(x),G_{n}(x),\ U_{n}(x)$ and$\ J_{n}%
(x)\ $satisfy the following three-terms recurrence relations:%
\begin{align}
2(n+1)F_{n+1}(x)  &  =(2n+1)(1+(1-2x)^{2})F_{n}(x)-2n(1-2x)^{2}F_{n-1}%
(x),\label{2.10}\\
n(1+2x)^{2}\!G_{n+1}(x)\!  &  =\!(2n-1)\! \! \left(  \!1+2x+2x^{2}\! \right)
\!G_{n}(x)\!-\!(n-1)G_{n-1}(x),\label{2.11}\\
(n+1)(1+t)^{2}U_{n+1}(t)  &  =(2n+1)(t^{2}+1)U_{n}(t)-n(1-t)^{2}%
U_{n-1}(t),\label{2.12}\\
(n+1)(1+t)^{2}J_{n+1}(t)  &  =(2n+1)(t^{2}+1)J_{n}(t)-n(1-t)^{2}J_{n-1}(t).
\label{2.13}%
\end{align}

\end{theorem}

\begin{remark}
\label{rem2.2} According to (\ref{2.12}) and (\ref{2.13}), $U_{n}(x)$ and
$J_{n}(x)\ $satisfy the same recurrence relation. In fact, from \cite[(49),
(55)]{R1} we have%
\begin{align}
U_{n}(x)  &  =4^{-n}\dbinom{2n}{n}%
{\textstyle \sum \limits_{k=0}^{n}}
\dbinom{n}{k}^{2}\dbinom{2n}{2k}^{-1}\left(  \frac{x-1}{x+1}\right)
^{2k},\label{2.14}\\
J_{n}(x)  &  =4^{-n}%
{\textstyle \sum \limits_{k=0}^{n}}
\frac{(2k)!(2n-2k)!}{k!^{2}(n-k)!^{2}}\left(  \frac{1-x}{1+x}\right)  ^{2k+1}.
\label{2.15}%
\end{align}
From (\ref{2.14}) and (\ref{2.15}) it is easy to deduce that%
\begin{equation}
J_{n}(x)=\frac{1-x}{1+x}U_{n}(x). \label{2.16}%
\end{equation}

\end{remark}

\begin{remark}
\label{rem2.3} From \cite[(56)]{R1} and \cite[(21)]{BMR} we know that%
\begin{align}
G_{n}(x) &  =4^{1-n}%
{\textstyle \sum \limits_{k=0}^{n-1}}
\dbinom{2k}{k}\dbinom{2n-2k-2}{n-k-1}(2x+1)^{-2k-1},\label{2.17}\\
F_{n}(x) &  =%
{\textstyle \sum \limits_{k=0}^{n}}
(-1)^{k}\dbinom{n}{k}\dbinom{2k}{k}(x(1-x))^{k}.\label{2.18}%
\end{align}
So, the recurrence relations (\ref{2.10})-(\ref{2.13}) are accompanied by%
\begin{align*}
F_{0}(x) &  =1,\ F_{1}(x)=1-2x+2x^{2};\\
G_{1}(x) &  =\frac{1}{2x+1},\ G_{2}(x)=\frac{1+2x+2x^{2}}{(2x+1)^{3}};\\
U_{0}(x) &  =1,\ U_{1}(x)=\frac{1+x^{2}}{(1+x)^{2}};\\
J_{0}(x) &  =\frac{1-x}{1+x},\ J_{1}(x)=\frac{(1-x)(1+x^{2})}{(1+x)^{3}}.
\end{align*}

\end{remark}

\section{Inequalities for indices of coincidence}

\begin{theorem}
\label{th3.1} The function $F_{n}(x)$ satisfies the inequalities%
\begin{equation}
F_{n+1}(x)\leq \frac{1+(4n-2)x(1-x)}{1+(4n+2)x(1-x)}F_{n-1}(x), \label{3.1}%
\end{equation}
and%
\begin{equation}
F_{n}(x)\leq \frac{1+4nx(1-x)}{1+(4n+2)x(1-x)}F_{n-1}(x), \label{3.2}%
\end{equation}
for all $n\geq1,\ x\in \lbrack0,1].$
\end{theorem}

\begin{proof}
We start with (see \cite[(29)]{R1})%
\[
F_{n}(x)=\frac{1}{\pi}\int_{0}^{1}f^{n}(x,t)\frac{dt}{\sqrt{t(1-t)}},
\]
where $f(x,t):=t+(1-t)(1-2x)^{2}\in \lbrack0,1].$

It follows that%
\begin{equation}
F_{n+1}(x)\leq F_{n}(x).\label{3.3}%
\end{equation}
On the other hand,%
\begin{align*}
&F_{n-1}(x)+F_{n+1}(x)-2F_{n}(x)=\\
&=\frac{1}{\pi}\int_{0}^{1}f^{n-1}(x,t)\left[
1+f^{2}(x,t)-2f(x,t)\right]\!\!\frac{dt}{\sqrt{t(1-t)}},
\end{align*}
which entails%
\begin{equation}
2F_{n}(x)\leq F_{n-1}(x)+F_{n+1}(x).\label{3.4}%
\end{equation}
According to (\ref{2.10}), we have%
\begin{equation}
F_{n}(x)=a_{n}(x)F_{n-1}(x)+b_{n}(x)F_{n+1}(x),\label{3.5}%
\end{equation}
where%
\[
a_{n}(x)=\frac{n(1-2x)^{2}}{(2n+1)(1-2x+2x^{2})},\ b_{n}(x)=\frac
{n+1}{(2n+1)(1-2x+2x^{2})}.
\]
Using (\ref{3.4}) and (\ref{3.5}) we get%
\[
2a_{n}(x)F_{n-1}(x)+2b_{n}(x)F_{n+1}(x)\leq F_{n-1}(x)+F_{n+1}(x),
\]
which yields%
\[
(2b_{n}(x)-1)F_{n+1}(x)\leq(1-2a_{n}(x))F_{n-1}(x),
\]
and this immediately leads to (\ref{3.1}). To prove (\ref{3.2}) it suffices to
combine (\ref{3.4}) and (\ref{3.1}).
\end{proof}

\begin{theorem}
\label{th.3.2} The following inequalities hold:%
\begin{align}
2U_{n}  &  \leq U_{n-1}+U_{n+1},\label{3.6}\\
U_{n+1}(x)  &  \leq \frac{1+4nx+x^{2}}{1+(4n+4)x+x^{2}}U_{n-1}(x),\label{3.7}\\
U_{n}(x)  &  \leq \frac{1+(4n+2)x+x^{2}}{1+(4n+4)x+x^{2}}U_{n-1}(x),
\label{3.8}\\
2G_{n}  &  \leq G_{n-1}+G_{n+1},\label{3.9}\\
G_{n+1}(x)  &  \leq \frac{1+(4n-2)x(1+x)}{1+(4n+2)x(1+x)}G_{n-1}%
(x),\label{3.10}\\
G_{n}(x)  &  \leq \frac{1+4nx(x+1)}{1+(4n+2)x(x+1)}G_{n-1}(x),\label{3.11}\\
2J_{n}  &  \leq J_{n-1}+J_{n+1},\label{3.12}\\
J_{n+1}(x)  &  \leq \frac{1+4nx+x^{2}}{1+(4n+4)x+x^{2}}J_{n-1}(x),\label{3.13}%
\\
J_{n}(x)  &  \leq \frac{1+(4n+2)x+x^{2}}{1+(4n+4)x+x^{2}}J_{n-1}(x).
\label{3.14}%
\end{align}

\end{theorem}

\begin{proof}
\bigskip The proof is similar to that of Theorem \ref{th3.1}, starting from
(see \cite[(48), (58), (63)]{R1}):%
\begin{align*}
U_{n}(x)  &  =\frac{1}{\pi}\int \nolimits_{0}^{1}\left(  t+(1-t)\left(
\frac{1-x}{1+x}\right)  ^{2}\right)  ^{n}\frac{dt}{\sqrt{t(1-t)}},\\
G_{n}(x)  &  =\frac{1}{\pi}\int \nolimits_{0}^{1}\left(  t+(1-t)(1+2x)^{2}%
\right)  ^{-n}\frac{dt}{\sqrt{t(1-t)}},\\
J_{n}(x)  &  =\frac{1}{\pi}\int \nolimits_{0}^{1}\left(  t+(1-t)\left(
\frac{1+x}{1-x}\right)  ^{2}\right)^{-n-1}\!\!\!\!\frac{dt}{\sqrt{t(1-t)}}.
\end{align*}

\end{proof}

\section{Other inequalities}

Let us return to the index of coincidence (\ref{1.1}). According to
\cite[(10)]{R1} we have for $c\neq0,$%
\[
S_{n,c}(x)=\frac{1}{\pi}\int \nolimits_{0}^{1}\left[  t+(1-t)(1+2cx)^{2}%
\right]  ^{-\frac{n}{c}}\frac{dt}{\sqrt{t(1-t)}}.
\]
Let $c<0$. Using Chebyshev's inequality for synchronous functions we can write%
\begin{align*}
S_{n-c,c}(x)  &  =\frac{1}{\pi}\!\!\int \nolimits_{0}^{1}\!\!\left[  t\!+\!(1\!-\!t)(1\!+\!2cx)^{2}%
\right]  ^{-\frac{n}{c}}\!\!\left[  t\!+\!(1\!-\!t)(1\!+\!2cx)^{2}\right]  \frac{dt}%
{\sqrt{t(1\!\!-\!\!t)}}\\
&  \geq \frac{1}{\pi}\int \nolimits_{0}^{1}\left[  t+(1-t)(1+2cx)^{2}\right]
^{-\frac{n}{c}}\frac{dt}{\sqrt{t(1-t)}}\cdot \\
& \quad \cdot \frac{1}{\pi}\int \nolimits_{0}^{1}\left[  t+(1-t)(1+2cx)^{2}%
\right]  \frac{dt}{\sqrt{t(1-t)}}\\
&  =S_{n,c}(x)(1+2cx+2c^{2}x^{2}).
\end{align*}

For $c>0$ we use Chebyshev's inequality for asynchronous functions and get the
reverse inequality. So we have

\begin{theorem}
\label{th.4.1}If $c<0$, then%
\begin{equation}
S_{n-c,c}(x)\geq(1+2cx(1+cx))S_{n,c}(x). \label{4.1}%
\end{equation}
If $c>0$, the inequality is reversed.
\end{theorem}

\begin{corollary}
\label{cor4.2} For $c=-1,$ (\ref{4.1}) and (\ref{3.2}) yield%
\[
(1-2x(1-x))F_{n}(x)\leq F_{n+1}(x)\leq \frac{1+(4n+4)x(1-x)}{1+(4n+6)x(1-x)}%
F_{n}(x).
\]
For $c=1$ we obtain%
\[
\frac{\ 1}{1+2x(1+x)}G_{n}(x)\leq G_{n+1}(x)\leq \dfrac{1+(4n+4)x(1+x)}%
{1+(4n+6)x(1+x)}G_{n}(x).
\]

\end{corollary}

Now using \cite[(48)]{R1} we have%
\begin{align*}
U_{n+1}(x)  &  =\frac{1}{\pi}\int \nolimits_{0}^{1}\left(  \left(  \frac
{1-x}{1+x}\right)  ^{2}+\frac{4x}{(1+x)^{2}}t\right)  ^{n}\cdot\\
&\quad\cdot\left(  \left(
\frac{1-x}{1+x}\right)  ^{2}+\frac{4x}{(1+x)^{2}}t\right)  \frac{dt}%
{\sqrt{t(1-t)}}\\
&  \geq U_{n}(x)\frac{1}{\pi}\int \nolimits_{0}^{1}\left(  \left(  \frac
{1-x}{1+x}\right)  ^{2}+\frac{4x}{(1+x)^{2}}t\right)  \frac{dt}{\sqrt{t(1-t)}%
}\\
&  =\frac{1+x^{2}}{(1+x)^{2}}U_{n}(x).
\end{align*}
Therefore, using also (\ref{3.8}), we get%
\begin{equation}
\frac{1+x^{2}}{(1+x)^{2}}U_{n}(x)\leq U_{n+1}(x)\leq \dfrac{1+(4n+6)x+x^{2}%
}{1+(4n+8)x+x^{2}}U_{n}(x). \label{4.2}%
\end{equation}

Now (\ref{4.2}) and (\ref{2.16}) yield%
\[
\frac{1+x^{2}}{\left(  1+x\right)  ^{2}}J_{n}(x)\leq J_{n+1}(x)\leq
\dfrac{1+(4n+6)x+x^{2}}{1+(4n+8)x+x^{2}}J_{n}(x).
\]

\begin{remark}
\label{rem4.3} Inequalities involving the function $K_{n}(x):=S_{n,0}(x)$ can
be obtained with different techniques and will be presented elsewhere.
\end{remark}

\begin{remark}
\label{rem4.4}All the above inequalities can be used to get information
concerning the entropies described in (\ref{1.2}) and (\ref{1.3}). We omit the details.
\end{remark}

\begin{remark}
\label{rem4.5} Convexity properties of the indices of coincidence and the
associated entropies were presented in \cite{R2, R3} but the hypothesis
$a_{n-k}=a_{k},\ k=0,1,\ldots,n$ was inadvertently omitted in \cite[Conjecture
6.1]{R3}.
\end{remark}


\begin{thebibliography}{9}                                                                                                %


\bibitem {A}M. Abramowitz, I.A. Stegun, Handbook of Mathematical Functions
with Formulas, Graphs, and Mathematical Tables, Dover Publications, Inc., New
York, (1970).

\bibitem {BMR}A. B\u{a}rar, G. Mocanu, I. Ra\c sa, Heun functions related to
entropies, RACSAM, 113(2019), 819--830.

\bibitem {R1}I. Ra\c{s}a, Entropies and Heun functions associated with
positive linear operators, Appl. Math. Comput., 268(2015), 422--431.

\bibitem {R2}I. Ra\c{s}a, Convexity properties of some entropies, Results
Math., 73:105 (2018).

\bibitem {R3}I. Ra\c{s}a, Convexity properties of some entropies (II), Results
Math., 74:154 (2019).

\end{thebibliography}
\end{document}